\documentclass[12pt, two]{article}
\usepackage{}

\usepackage{CJK}
\usepackage{amsthm}
\usepackage{amsmath,amssymb,mathrsfs}
\usepackage{amsfonts}
\usepackage{yhmath}
\usepackage{graphics}
\usepackage{amssymb}
\usepackage{epsfig}
\usepackage{latexsym}
\usepackage{indentfirst}
\allowdisplaybreaks
\usepackage[top=1in, bottom=1in, left=1.25in, right=1.25in]{geometry}
\numberwithin{equation}{section}
\usepackage{CJK}

\newtheorem{theorem}{Theorem}[section]
\newtheorem{lemma}{Lemma}[section]

\newtheorem{corollary}{Corollary}[section]
\newtheorem{remark}{Remark}[section]

\numberwithin{equation}{section} %¹«Ê½Ëæ½Ú±àºÅ

\title{
Spectral Bernstein theorems for submanifolds in Euclidean spaces
\footnotetext[0]{2010 Mathematics Subject Classification. Primary: 58J50, 53C20, 53C42. }
\footnotetext[0]{ $^{*}$Supported by NSFC grant No. 12171091.}
\footnotetext[0]{  ${}^{\dag}$ Partially supported by NSFC grant No. 11831005.}
\footnotetext[0]{{\em Key words and phrases.  Essential spectrum, Submanifolds.}}
}
\author{Yuxin Dong$^{*}$, Hezi Lin$^{\dag}$, Wei Zhang }
\date{}
\begin{document}
%\begin{CJK*}{GBK}{kai}

\maketitle

\begin{abstract}
In this paper, we consider the essential spectrum of submanifolds in Euclidean spaces under various geometric hypotheses. Our results involve extrinsic conditions such as finite total mean curvature, the convergence of the gradient of the extrinsic distance, and the extrinsic volume growth or the pinching curvature. In particular, we prove that the essential spectrum of a complete non-compact submanifold $M^n$ in a Euclidean space is $[0, +\infty)$ provided the second fundamental form $A$ of $M^n$ satisfies $\|A\|_{L^p} < \infty$,  $p>n$.
\end{abstract}

\section{Introduction}
The Laplace operator $-\triangle$ on a complete non-compact manifold $(M, g)$ is essentially self-adjoint on the domain $C^{\infty}_0(M)$, and thus it admits a unique self-adjoint extension to $L^2(M)$. The spectrum of $-\triangle$ consists of all $\lambda\geq 0$ such that $\triangle + \lambda I$ does not have bounded inverse.
Sometimes we denote it by $\sigma(M)$, which can be decomposed as
\begin{equation*}
\sigma(M)=\sigma_p(M)\cup \sigma_{ess}(M),
\end{equation*}
where $\sigma_p(M)$ is the point spectrum, formed by eigenvalues with finite multiplicity, and $\sigma_{ess}(M)$ is the essential spectrum, formed by accumulation points of the spectrum and by the eigenvalues with infinite multiplicity. Thus the essential spectrum is a closed subset of $[0, +\infty)$. The spectrum of $-\triangle$ on a complete non-compact manifold $M$ has been proved to be closely related to the geometric and topological properties of $M$.

Since the Euclidean space $\mathbb{R}^{m}$ has only essential spectrum consisting of the whole half line $[0, +\infty)$, we may naturally consider which extrinsic conditions ensure a submanifold of $\mathbb{R}^{m}$ having a trivial spectrum, that is, which submanifold has the same spectrum as $\mathbb{R}^{m}$. On the other hand, it is known that the essential spectrum of the Laplacian on a Riemannian manifold depends on the asymptotic behavior of the metric. Motivated by this fact, we intend to study the essential spectrum of submanifolds under global conditions.
In this article, we  consider the essential spectrum of a submanifold in Euclidean space without minimality condition.

Let $M^n$ be a complete $n$-dimensional submanifold in $\mathbb{R}^{m}$. Let $\rho(x)=\mbox{dist}_{\mathbb{R}^{m}}(x, o)$ be the distance function from some reference origin $o\in M$ in $\mathbb{R}^{m}$, and denote by $\rho$ the restriction $\rho|_{M^n}: M^n \rightarrow [0, +\infty)$.  This restriction is then called the extrinsic distance function from $o$ in $M^n$.
Denote
\begin{align*}
D(r)=\{x \in M: 0\leq \rho(x)\leq r\} \ \ \mbox{and}\ \  V_M(r)=\mbox{vol}(D(r)).
\end{align*}
If for any $\epsilon>0$, there exists a constant $c(\epsilon)$ such that
\begin{equation*}
V(r) \leq c(\epsilon) e^{\epsilon r}V(1)
\end{equation*}
for any $r>0$, then we say that the extrinsic volume of $M$ grows uniformly sub-exponentially.
Let $r(x)=d(x, o)$ be the intrinsic distance function from $o\in M$. Then $r(x)$ is a Lipschitz function and is smooth on $M\setminus \{o, \mbox{cut}(o)\}$, where $\mbox{cut}(o)$ is the cut locus of $o$. It is well known that $\triangle r$ exists on $M\setminus \{o\}$ in the sense of distribution and $\triangle r$ is locally integrable away from $o$ (see \cite{CL}).
Denote by
\begin{align*}
B(t)=B_o(t)\ \ \mbox{and}\ \ V(t)=\mbox{vol}(B_o(t))
\end{align*}
the geodesic ball of radius $t$ at $o$ and its volume respectively.

Our first main result explores the essential spectrum of submanifolds with finite total mean curvature. By employing the extrinsic distance function, we construct approximate eigenfunctions and obtain the following result.
\begin{theorem} (cf. Theorem \ref{thm-1})
Let $M^n$ be a complete non-compact properly immersed submanifold in $\mathbb{R}^{m}$ with $\|H\|_{L^n} < \infty$. Assume that $|\bigtriangledown \rho| \rightarrow \gamma$ as $\rho \rightarrow \infty$ for some $\gamma >0$ and the extrinsic volume of $M$ grows uniformly sub-exponentially. Then the $L^2$ essential spectrum of $M$ is $[0, +\infty)$.
\end{theorem}

By constructing approximate eigenfunctions using the intrinsic distance function, we prove the following spectral Bernstein theorem under the assumption that the submanifold is asymptotically totally umbilical in some average sense at infinity.
\begin{theorem} (cf. Theorem \ref{thm-2})
Let $M^n$ be a complete non-compact submanifold in $\mathbb{R}^{m}$. Suppose the traceless second fundamental form $\Phi$ of $M$ satisfies
\begin{align*}
\underset{t \rightarrow \infty}\lim \frac{1}{V(t)}
\int_{B(t)} |\Phi|^p dv=0  \label{ft}
\end{align*}
for some $p>n$. Then the $L^2$ essential spectrum of $M$ is $[0, +\infty)$.
\end{theorem}
Cavalcante,  Mirandola and Vit\'{o}rio \cite{CMV} proved that if a complete non-compact manifold $M^n$ isometrically immerses in a manifold with bounded geometry, and $\|H\|_{L^{p}} < \infty$ for some $n \leq p \leq +\infty$, then $M$ has infinite volume. Using this result, we prove the following  spectral Bernstein theorem, only by assuming $L^p$ integral condition on the second fundamental form $A$.
\begin{corollary} (cf. Corollary \ref{total-mean})
Let $M^n$ be a complete non-compact submanifold in $\mathbb{R}^{m}$ with $\|A\|_{L^p} < \infty$ for some $p>n$. Then the $L^2$ essential spectrum of $M$ is $[0, +\infty)$.
\end{corollary}

There is an abundant literature studying the spectrum of complete Riemannian manifolds under various geometric conditions. We briefly overview some of the main achievements.
Donnelly \cite{Do} showed
\begin{equation*}
\sigma_{ess}(M)=\Big[\frac{(n-1)^2c^2}{4}, +\infty)
\end{equation*}
for a Hadamard manifold $M^n$ whose sectional curvature approaches $-c^2$ at infinity.
In 1997, Kumura \cite{Ku1} established the following result: let $r$ be the distance from a pole (i.e., a point $x \in M$ where the exponential map $exp_x: T_xM\rightarrow M$ is a diffeomorphism), then $\sigma_{ess}(M)=[\frac{c^2}{4}, +\infty)$ provided
\begin{equation*}
\underset{n\rightarrow \infty}{\lim}\underset{r\geq n}{\sup}|\triangle r-c|=0.
\end{equation*}
On the other hand, via a modification of Petersen-Wei's integral inequality \cite{PW},  Kumura \cite{Ku2} weakened the curvature requirements, by replacing the pointwise pinching on the sectional curvature with a combination of a lower bound on a suitably weighted volume and an $L^p$-bound on the Ricci curvature.
Without assuming the existence of a pole, Wang \cite{Wa} showed that if
\begin{equation*}
\mbox{Ric}(x)\geq -\delta(n)\frac{1}{r^2(x)}
\end{equation*}
where $\delta(n)$ is a small constant depending on the dimension and $r(x)$ is the distance to a fixed point, then the $L^p$ essential spectrum is $[0, +\infty)$ for all $p\in [1, +\infty]$, by using a result proved by Sturm \cite{St}. This work was extended by Lu and Zhou \cite{LZ} who proved that  the $L^p$ essential spectrum  is $[0, +\infty)$ under the assumption that
\begin{equation*}
\underset{x \rightarrow \infty}{\liminf} \mbox{Ric}_M(x)=0.
\end{equation*}

It would be interesting to understand the relations between the spectrum and the extrinsic geometry of immersed submanifolds in various spaces. The relevant geometric restrictions are related to pointwise or integral curvature bounds of the ambient spaces, and the second fundamental form of the submanifolds. When the submanifolds have extrinsic bounds, their spectrums have been studied by many mathematicians. Bessa, Jorge and Montenegro \cite{BJM} proved that if a minimal submanifold is properly immersed in a ball of $\mathbb{R}^{m}$, then there exists only pure point spectrum. Bessa, Jorge and Mari \cite{BJMa} investigated some of the relations between the spectrum of a non-compact, extrinsically bounded submanifold and the Hausdorff dimension of its limit set.

The situation for unbounded submanifolds will, however, be different in general. P. Castillon \cite{Ca} determined the essential spectrum of the Laplacian and the stability operator of a submanifold in the hyperbolic space with constant mean curvature $|H|<1$ and finite total curvature. In \cite{FS}, Freitas and Salavessa showed that the essential spectrum of a complete minimal properly immersed hypersurface in $\mathbb{R}^{m}$ consists of the half line $[0, +\infty)$, under a volume growth condition on extrinsic balls and a condition on the unit normal at infinity. Lima, Mari, Montenegro and Vieira \cite{LMMV} characterized $\sigma_{ess}(M)$ for a minimal properly immersed submanifold $M$ in $N$ which has a pointwise or an integral pinching to a space form, under a mild volume assumption. In their proof, the minimality condition enters in a crucial and subtle way.

The paper is organized as follows. In Section 2, we give some necessary notations for submanifold geometry, and recall  Charalambous-Lu's criterion. For later use, we prove smooth approximations of distance function in the integral sense, which is suit for integral curvature assumption. In Section 3, we prove a spectral Bernstein theorem for submanifolds without minimality condition. In Section 4, we present a sequence of lemmas which involve volume estimate and the integral decay of the Laplacian of distance function. Thanks to these estimates we can finally prove the main theorem by  Charalambous-Lu's criterion.

\section{Preliminaries}
Let $M^n$ be a complete $n$-dimensional submanifold in an $m$-dimensional Riemannian manifold $N^{m}$. Fix a point $x \in M^n$ and a local orthonormal frame $\{e_1, \cdots, e_{m}\}$ of $N^{m}$ such that $\{e_1, \cdots, e_{n}\}$ are tangent fields of $M^n$ at $x$. Denote by $\nabla$ and $\overline{\nabla}$ the Riemannian connection of $M^n$ and $N^{m}$ respectively.
For each $\alpha$, $n+1\leq \alpha \leq m$, define a linear map $A_{\alpha}: T_xM \rightarrow T_xM$ by
\begin{equation*}
\langle A_{\alpha} X, Y\rangle= \langle \overline{\nabla}_XY, e_{\alpha}\rangle,
\end{equation*}
where $X, Y\in T_xM$. Let $A: T_xM \times T_xM\rightarrow (T_xM)^\perp$ denote the second fundamental form of $M$ defined by
\begin{equation*}
A(X, Y)= \overset{m}{\underset{\alpha=n+1}{\sum}}\langle A_\alpha X, Y\rangle e_\alpha.
\end{equation*}
Denote by $h^\alpha_{ij} =\langle A_{\alpha} e_i, e_j\rangle $ the components of the second fundamental form.
The squared norm $|A|^2$ of the second fundamental form and the mean curvature vector $H$  are defined by:
\begin{equation*}
|A|^2=\underset{\alpha}{\sum}tr(A_\alpha^2)=\underset{\alpha, i,j}{\sum}(h^\alpha_{ij})^2 \ \ \  \ \ \
\mbox{and} \ \ \ \ \ \  H=\frac{1}{n}\underset{\alpha}{\sum}tr(A_\alpha)e_{\alpha}=\frac{1}{n}\underset{\alpha, i}{\sum}h^\alpha_{ii}e_{\alpha}
\end{equation*}
respectively. The traceless second fundamental form  is defined by
\begin{equation*}
\Phi=\overset{m}{\underset{\alpha=n+1}{\sum}}(A^\alpha_{ij}-  H^\alpha g_{ij}) e_\alpha.
\end{equation*}
Its norm satisfies
\begin{equation*}
|\Phi|^2=|A|^2-n|H|^2,
\end{equation*}
which measures how much the immersion deviates from being totally umbilical.
The Ricci curvature of $M$ is given by the Gauss equation
\begin{equation*}
 R_{ijkl}= \bar{R}_{ijkl} + h^\alpha_{ik}h^\alpha_{jl}-h^\alpha_{il}h^\alpha_{jk},
\end{equation*}
where $R_{ijkl}$ and $\bar{R}_{ijkl}$ are the components of the curvature tensors of $M^n$ and $N^{m}$, respectively. We say that  $N^{m}$ has nonnegative $(n-1)$-th Ricci curvature if for any $x \in N$ and any $n$ orthonormal vectors $\{e, e_1, \cdots, e_{n-1}\} \subset T_x N$, the curvature tensor
satisfies $\sum_{i=1}^{n-1}\langle R(e_i, e)e, e_i\rangle \geq 0$.

The following lemma gives a lower bound estimate for the Ricci curvature of a submanifold in a manifold which has nonnegative $(n-1)$-th Ricci curvature.
\begin{lemma}\label{Lin} (\cite{Lin})
Let $M^n$  be a submanifold immersed in a Riemannian manifold $N^{m}$ which has nonnegative $(n-1)$-th Ricci curvature. Then the Ricci curvature of $M^n$ satisfies
\begin{equation}
\mbox{Ric} \geq - \frac{n}{4}|\Phi|^{2}. \label{Ric-1}
\end{equation}
Assume
$(n-1) |A|^2 \leq n^2|H|^2$.
 Then
$\mbox{Ric} \geq 0$.
\end{lemma}

It is known that the essential spectrum of a Riemannian manifold is closely related to its volume growth.  Let us recall  the following result proved by Sturm.
\begin{lemma} \label{Stc} (\cite{St})
Let $M$ be a complete non-compact manifold whose Ricci curvature has a
lower bound. If the volume of $M$ grows uniformly sub-exponentially, then the $L^p$ spectrum are the same for all $p \in [1,\infty]$.
\end{lemma}
N. Charalambous and Z. Lu \cite{CL} proved an important generalization of Weyl's criterion, which proves to be a powerful tool in expanding the set of manifolds for which the $L^2$ essential spectrum is the nonnegative real line.
\begin{lemma}\label{criteria} (\cite{CL})
A number $\lambda \in \mathbb{R^+}$ lies in $\sigma_{\mbox{ess}}(M)$ if for any $\epsilon>0$ and any compact subset $K$ of $M$, there exists a nonzero function $u$ in the domain of $\triangle$ satisfying
\begin{equation}
\|u\|_{L^{\infty}}\|\triangle u + \lambda u\|_{L^1} \leq \epsilon \|u\|^2_{L^2}
\end{equation}
and that the support of $u$ is outside $K$.
\end{lemma}

Most of the approximate eigenfunctions we can construct explicitly must be related to the distance function $r(x)=d(x, o)$ from a reference point $o\in M$. By mollifying the  function $r(x)$ as in \cite{LZ, CL}, we have a smooth function $\widetilde{r}(x)$ approximating $r(x)$ and satisfying the following pointwise and integral properties.
\begin{lemma} \label{r-smooth}
Let  $\eta: \mathbb{R}^+ \rightarrow \mathbb{R}^+$ be a positive continuous decreasing function satisfying
$\underset{r\rightarrow \infty}{\lim} \eta (r) =0$, then there exists a smooth function $\widetilde{r}(x)$ on $M$ such that\\
 (a) for  any $x \in M$ with $r(x) >2$, we have $|\widetilde{r}(x) - r(x)| \leq \eta (r(x))$ and $|\nabla \widetilde{r}(x)| \leq 3$;\\
 (b) $\|\nabla \widetilde{r}-\nabla r\|_{L^1(M\setminus B(R))} \leq \eta(R)$ for any $R>0$;\\
(c) for any $b>a>1$, we have
\begin{equation*}
\int_{B(b)\setminus B(a)} |\triangle \widetilde{r}| dv \leq \int_{B(b)\setminus B(a)} |\triangle r|dv +2\eta(a).
\end{equation*}
\end{lemma}
\begin{proof}
 Taking $x \in M$, we consider a neighborhood $U_x \subset B_x(1)$ such that\\
(i) there exists local coordinates $x_i$ defined in an open set containing $\overline{U_x}$ such that $x=(0,\ldots, 0)$;\\
(ii) Choose a locally finite cover $\{U_i\}_{i \in \mathbb{N}}$ from $\{U_x\}_{x \in M}$, and let
$\{\psi_i\}$ be the associated partition of the unity.

Let $\xi(x) \geq 0$ be a  smooth function whose support is within the unit ball of $\mathbb{R}^n$ and satisfying $\int_{\mathbb{R}^n}\xi dx =1$.
Fix $i \in \mathbb{N}$ and denote by $\mathbf{x}_i=(x_i^1,\ldots, x_i^n)$ the local coordinate of $U_i$.
Define
\begin{equation}
r_{i, \epsilon_i}(\mathbf{x}_i)=\frac{1}{\epsilon_i^n} \int_{\mathbb{R}^n}\xi\left(\frac{\mathbf{y}_i}{\epsilon_i}\right)r(\mathbf{x}_i +\mathbf{y}_i)d\mathbf{y}_i
\end{equation}
with $\epsilon_i>0$, and $\Lambda_i = \sup_{U_i}\{|\nabla \psi_i|+|\triangle \psi_i|\}+1>1$. Choose $\epsilon_i<1$ small enough such that, for all $x \in U_i$,
\begin{equation*}
|r_{i, \epsilon_i}(x)- r(x)|\leq \frac{\eta(r(x))}{2^{i} \Lambda_i},\ \ \ \
\|r_{i, \epsilon_i}(x)- r(x)\|_{L^1(U_i)}\leq \frac{\eta(r(x))}{2^{i+1} \Lambda_i}.
\end{equation*}
Since
\begin{equation*}
\nabla r_{i, \epsilon_i}(\mathbf{x}_i)=\frac{1}{\epsilon_i^n} \int_{\mathbb{R}^n}\xi\left(\frac{\mathbf{y}_i}{\epsilon_i}\right)\nabla r(\mathbf{x}_i +\mathbf{y}_i)d\mathbf{y}_i,
\end{equation*}
and
\begin{equation*}
\triangle r_{i, \epsilon_i}(\mathbf{x}_i)=\frac{1}{\epsilon_i^n} \int_{\mathbb{R}^n}\xi\left(\frac{\mathbf{y}_i}{\epsilon_i}\right)\triangle r(\mathbf{x}_i +\mathbf{y}_i)d\mathbf{y}_i
\end{equation*}
in the sense of distribution, and $\nabla r$, $\triangle r$ are locally integrable, we can choose $\epsilon_i<1$ small enough such that, for all $x \in U_i$,
\begin{equation}
\ \ \ \  |\nabla r_{i, \epsilon_i}(x)| <2,\ \ \ \   \|\nabla r_{i, \epsilon_i}(x)- \nabla r(x)\|_{L^1(U_i)}\leq \frac{\eta(r(x))}{2^{i+1} \Lambda_i} \label{gr}
\end{equation}
and
\begin{equation}
\|\triangle r_{i, \epsilon_i}(x)- \triangle r(x)\|_{L^1(U_i)}\leq \frac{\eta(r(x))}{2^{i+1} \Lambda_i}. \label{Lap}
\end{equation}
Define, for $x \in M$,
\begin{equation*}
\tilde{r}(x)=\underset{i}{\sum}\psi_i(x)  r_{i,\epsilon_i}(x).
\end{equation*}
Then
\begin{align*}
|\widetilde{r}(x) - r(x)| =&|\underset{i}{\sum}\psi_i(x)  (r_{i,\epsilon_i}(x)-r(x))|\\
\leq& \underset{i}{\sum}|\psi_i(x)|  |r_{i,\epsilon_i}(x)-r(x)|
\leq \eta (r(x)).
\end{align*}
Without loss of generality, we assume that $\eta (r)\leq 1$. Noticing that $\underset{i}{\sum}\nabla \psi_i(x)  r(x)=\nabla (\underset{i}{\sum}\psi_i(x))  r(x)=0$ almost everywhere on $M$, we have
\begin{align*}
|\nabla \widetilde{r}|=\big{|}\underset{i}{\sum}\big{(}\psi_i  \nabla r_{i,\epsilon_i}+\nabla \psi_i  (r_{i,\epsilon_i}-r)\big{)}\big{|}
\leq& \underset{i}{\sum}\psi_i  |\nabla r_{i,\epsilon_i}| + \underset{i}{\sum}|\nabla \psi_i||r_{i,\epsilon_i}-r|\\
\leq & 2\underset{i}{\sum}\psi_i + \underset{i}{\sum} \Lambda_i \frac{\eta(r(x))}{2^{i}\Lambda_i}
\leq 3.
\end{align*}
Thus (a) follows. Similarly, we have
\begin{align*}
\nabla \widetilde{r}-\nabla r=\underset{i}{\sum}\big{(}\psi_i ( \nabla r_{i,\epsilon_i}-\nabla r)+\nabla \psi_i  (r_{i,\epsilon_i}-r)\big{)}.
\end{align*}
Hence,
\begin{align*}
\nonumber &\int_{M\setminus B(R)}|\nabla \tilde{r}-\nabla r| dv \\
\nonumber\leq& \underset{i}{\sum}\int_{M\setminus B(R)}\psi_i |\nabla  r_{i,\epsilon_i}-\nabla r|dv+ \underset{i}{\sum} \int_{M\setminus B(R)}|\nabla \psi_i||r_{i,\epsilon_i} -r| dv\\
\nonumber\leq& \underset{i}{\sum}\int_{(M\setminus B(R))\cap U_i} |\nabla  r_{i,\epsilon_i}-\nabla r| dv
+\underset{i}{\sum}  \Lambda_i \int_{(M\setminus B(R))\cap U_i}|r_{i,\epsilon_i}-r|dv\\
\nonumber\leq& \underset{i}{\sum}  \frac{\eta(R)}{2^{i+1}}+\underset{i}{\sum} \Lambda_i \frac{\eta(R)}{2^{i+1}\Lambda_i}\\
=&\eta(R).
\end{align*}

To prove (c), observing that
\begin{equation*}
\sum_i\triangle \psi_i=\triangle(\sum_i \psi_i)=\triangle 1=0\ \  \mbox{and}\ \
 \sum_i\langle \nabla \psi_i, \nabla r \rangle=\langle \nabla (\sum_i\psi_i), \nabla r \rangle =0,
\end{equation*}
we have
\begin{align}
\nonumber \triangle \tilde{r}-\triangle r=&\underset{i}{\sum}\psi_i (\triangle  r_{i,\epsilon_i}-\triangle r)+\underset{i}{\sum}(\triangle \psi_i)(r_{i,\epsilon_i} -r)\\
 &+2 \underset{i}{\sum}\langle \nabla \psi_i, \nabla r_{i,\epsilon_i}-\nabla r \rangle. \label{Le02}
\end{align}
By (\ref{Le02}),  (\ref{gr}) and (\ref{Lap}),  it follows that
\begin{align*}
\nonumber &\int_{B(b)\setminus B(a)}|\triangle \tilde{r}-\triangle r| dv \\
\nonumber\leq& \underset{i}{\sum}\int_{B(b)\setminus B(a)}\psi_i |\triangle  r_{i,\epsilon_i}-\triangle r|dv+ \underset{i}{\sum} \int_{B(b)\setminus B(a)}|\triangle \psi_i||r_{i,\epsilon_i} -r| dv\\
\nonumber&+2 \underset{i}{\sum}\int_{B(b)\setminus B(a)}|\nabla \psi_i||\nabla r_{i,\epsilon_i}-\nabla r|dv\\
\nonumber\leq& \underset{i}{\sum}\int_{(B(b)\setminus B(a))\cap U_i} |\triangle  r_{i,\epsilon_i}-\triangle r| dv
+\underset{i}{\sum} \Lambda_i \int_{(B(b)\setminus B(a))\cap U_i} |r_{i,\epsilon_i} -r| dv\\
\nonumber&+2\underset{i}{\sum}\Lambda_i \int_{(B(b)\setminus B(a))\cap U_i}|\nabla r_{i,\epsilon_i}-\nabla r|dv\\
\nonumber\leq& 2\eta(a).
\end{align*}
Therefore, we conclude
\begin{align*}
\int_{B(b)\setminus B(a)}|\triangle \tilde{r}| dv \leq& \int_{B(b)\setminus B(a)}|\triangle r| dv +\int_{B(b)\setminus B(a)}|\triangle \tilde{r}-\triangle r| dv\\
\leq& \int_{B(b)\setminus B(a)}|\triangle r| dv + 2\eta(a),
\end{align*}
which gives (c) and the Lemma is proved.
\end{proof}

\section{Essential spectrum of submanifolds with integral mean curvature}
Let $M^n$ be a complete  submanifold in $\mathbb{R}^{m}$, and let $\rho(x)$ denote the extrinsic distance of $x $ relative to a point $o\in M$. We have the following spectral theorem under extrinsic assumptions.

\begin{theorem} \label{thm-1}
Let $M^n$ be a complete non-compact properly immersed submanifold in $\mathbb{R}^{m}$ with $\|H\|_{L^n} < \infty$. Assume that $|\bigtriangledown \rho| \rightarrow \gamma$ as $\rho \rightarrow \infty$ for some $\gamma >0$ and the extrinsic volume of $M$ grows uniformly sub-exponentially. Then the $L^2$ essential spectrum of $M$ is $[0, +\infty)$.
\end{theorem}
\begin{proof}
Using the Gauss formula, a standard computation gives (see e.g. \cite{JK})
 \begin{equation*}
 \mbox{Hess}^M(\rho)(X,Y)=\mbox{Hess}^{\mathbb{R}^m} (\rho)(X,Y) + \langle A(X,Y), \overline{\nabla}\rho \rangle,
 \end{equation*}
which implies
\begin{equation*}
 \triangle \rho =\mbox{tr}_M(\mbox{Hess}^{\mathbb{R}^m}  \rho) + n\langle H, \overline{\nabla}\rho \rangle.
 \end{equation*}
Considering the fact that $\mbox{Hess}^{\mathbb{R}^m}  \rho = \frac{1}{\rho}(g_{\mathbb{R}^m}- d\rho\otimes d\rho)$ and  $|\overline{\nabla}  \rho|=1$, we have
\begin{equation}
 |\triangle \rho| \leq \frac{n}{\rho} + n|H|. \label{lap-H}
\end{equation}

Let $a$, $b$, $R>1$ be positive real numbers satisfying $b>a+2R$ and $a>2R$. Define a cut-off function $\phi$ satisfying $\phi\leq 1$, $\phi\equiv 1$ on $[a/R, b/R]$,  $\mbox{supp} \phi \subset [a/R-1, b/R+1]$ and $|\phi'| +|\phi''|$ is bounded. Denote $D(r)=\{x \in M: 0\leq \rho(x)\leq r\}$ and $V_M(r)=\mbox{volume}$ of $D(r)$.

For any given $\lambda >0$, consider
\begin{equation*}
\psi_{ab}= \phi\big(\frac{\rho}{R}\big)e^{i\sqrt{\lambda}\rho} \label{cut-off}.
\end{equation*}
Then
\begin{equation*}
\nabla \psi_{ab}= \Big( \frac{1}{R}\phi'\big(\frac{\rho}{R}\big) +i\sqrt{\lambda}\phi\big(\frac{\rho}{R}\big) \Big) e^{i\sqrt{\lambda}\rho}\nabla \rho.
\end{equation*}
Thus
\begin{align*}
\triangle \psi_{ab}+ \lambda \gamma^2\psi_{ab}=& \Big[ \frac{1}{R^2}\phi''\big(\frac{\rho}{R}\big)|\nabla \rho|^2 +i\sqrt{\lambda}\frac{2}{R}\phi'\big(\frac{\rho}{R}\big)|\nabla \rho|^2
+ \Big( \frac{1}{R}\phi'\big(\frac{\rho}{R}\big) \\ &+i\sqrt{\lambda}\phi\big(\frac{\rho}{R}\big)\Big)\triangle \rho \Big] e^{i\sqrt{\lambda}\rho}
+ \lambda (\gamma^2 -|\nabla \rho|^2)\psi_{ab}.
\end{align*}
Substituting (\ref{lap-H}) into the above inequality yields
\begin{align}
\nonumber |\triangle \psi_{ab}+ \lambda \gamma^2 \psi_{ab}|\leq & \Big[C_1 \Big(\frac{1}{R} + |H| \Big) + \lambda  |\gamma^2 -|\nabla \rho|^2| \Big]\chi_{ab},
\end{align}
where $\chi_{ab}$ is the characteristic function of $D(b+R)\setminus D(a-R)$ and $C_1$ is a constant depending only on $\lambda$ and $\phi$.
Integrating and using the H\"{o}lder inequality gives
\begin{align}
\nonumber \int_M |\triangle \psi_{ab}+ \lambda \gamma^2 \psi_{ab}| dv  \leq & \frac{C_1}{R}(V_M(b+R)-  V_M(a-R))+ C_1\int_{D(b+R)\setminus D(a-R)}|H| dv \\
\nonumber &+ \lambda \int_{D(b+R)\setminus D(a-R)} |\gamma^2 -|\nabla \rho|^2| dv\\
\nonumber \leq& \frac{C_1}{R}V_M(b+R) + \lambda  \underset{D(b+R)\setminus D(a-R)}{\sup}|\gamma^2 -|\nabla \rho|^2|V_M(b+R) \\
 &+C_1 \big(\int_{D(b+R) \setminus D(a-R)}|H|^n dv \big)^{\frac{1}{n}} \big(V_M(b+R) \big)^{1-\frac{1}{n}}. \label{FV}
\end{align}
For any $\epsilon>0$, by the assumptions $\underset{\rho\rightarrow \infty}{\limsup}|\nabla \rho|^2 = \gamma^2$ and $\int_M |H|^n dv < \infty$, there exists $a_0$ such that for any $a\geq a_0$ we have
\begin{equation}
\lambda  \underset{D(b+R)\setminus D(a-R)}{\sup}|\gamma^2 -|\nabla \rho|^2| \leq \frac{\epsilon}{3} \label{S1}
\end{equation}
and
\begin{equation}
C_1 \big(\int_{D(b+R) \setminus D(a-R)}|H|^n dv \big)^{\frac{1}{n}}\leq \frac{\epsilon}{3}. \label{S2}
\end{equation}
Since $\int_M |H|^n dv < \infty$, it follows from \cite{CMV} that the volume of $M$ is infinite. Hence the extrinsic volume of $M$ is also infinite by the properly immersion.  Then there exists $R_0$ and $b_0$ such that
\begin{equation}
\frac{C_1}{R} \leq \frac{\epsilon}{3} \ \ \ \mbox{and}\ \ \  V_M(b+R) \geq 1 \label{S3}
\end{equation}
holds for any $R \geq R_0$ and $b\geq b_0$. Substituting (\ref{S1}), (\ref{S2}) and (\ref{S3}) into (\ref{FV}) yields
\begin{align}
 \int_M |\triangle \psi_{ab}+ \lambda \gamma^2 \psi_{ab}| dv  \leq \epsilon V_M(b+R). \label{up-volume1}
\end{align}
On the other hand, we have
\begin{align}
\nonumber \int_M | \psi_{ab}| dv  \geq  V_M(b) - V_M(a).
\end{align}
Choose $b_1\geq b_0$  such that $ V_M(b_1) \geq \frac{1}{2}V_M(a)$. Then it gives
 \begin{align}
 \int_M | \psi_{ab}| dv  \geq  \frac{1}{2}V_M(b) \label{down-volume}
\end{align}
for any $b\geq b_1$. Using a contradiction argument as in \cite{LZ}, the uniformly sub-exponentially volume growth of $M$ implies that there exists a sequence $y_k \rightarrow \infty$ such that $V_M(y_k +R) \leq 2V_M(y_k)$. If not, then for a fixed number $y$ and for all $k \in \mathbb{N}$ we have
\begin{equation*}
V_M(y +kR) > 2^k V_M(y).
\end{equation*}
But by the assumption that the extrinsic volume of $M$ grows uniformly sub-exponentially,
\begin{equation*}
 2^k V_M(y) < V_M(y +kR) \leq C(\epsilon_0)V_M(1)e^{\epsilon_0(y +kR)}
\end{equation*}
for any $\epsilon_0>0$ and $k$ large.
This gives a contradiction when $\epsilon_0 R <\log 2$.
Therefore,  there is a constant $b\geq b_1$ such that $V_M(b+R) \leq 2V_M(b)$.
Combining with (\ref{up-volume1}) and (\ref{down-volume}) yields
\begin{align*}
 \int_M |\triangle \psi_{ab}+ \lambda \gamma^2 \psi_{ab}| dv  \leq 4\epsilon \int_M | \psi_{ab}| dv .
\end{align*}
According to Lemma \ref{criteria}, $\lambda \gamma^2$ belongs to the essential spectrum for any $\lambda > 0$. Therefore, the $L^2$ essential spectrum of $M$ is $[0, +\infty)$.
 \end{proof}

\section{Essential spectrum of submanifolds with finite total curvature}
Let $r(x)=d(x, o)$ be the distance function from a reference point $o\in M$.
Around $o$ using exponential polar coordinates, we can write the volume element as
\begin{equation*}
dv= \omega (r, \theta) dr\wedge d\theta_{n-1},
\end{equation*}
where $d\theta_{n-1}$ is the standard volume element on the unit sphere $S^{n-1}(1)$. As $r$ increases $\omega (r, \theta)$ becomes undefined but we can just declare it to be zero for those $r$. It is known that
\begin{equation*}
\omega' = h(r, \theta)\omega,
\end{equation*}
where $h(r, \theta)$ is $n$ times the mean curvature of the geodesics sphere $\partial B_o(r)$, which is equal to $\triangle r$ at smooth points of $r(x)$. We know that $h(r, \theta)$ satisfies
\begin{equation}
h'+ \frac{h^2}{n-1} \leq - \mbox{Ric}(\partial r, \partial r). \label{Rac}
\end{equation}
Let us define  $f_+ = \max \{0, f\}$ for a function $f$ and declare that $f_+$ is 0 whenever it becomes undefined. Based on the relation (\ref{Rac}), Petersen and Wei proved the following result.
\begin{lemma}\label{lem4.1} ([16, Lemma 2.2])
Let $M^n$ be a complete Riemannian manifold of dimension $n$.  Then for $q> \frac{n}{2}$, we have
\begin{align*}
\int^r_0 \{h(t, \theta)- h_\kappa(t)\}_+^{2q}\omega(t,\theta)dt \leq C_1(n,q)
\int_0^r \{(n-1)\kappa- \mbox{Ric}(\partial r, \partial r)\}_+^q \omega(t,\theta)dt,
\end{align*}
where $C_1(n,q)$ is a constant depending only on $n$ and $q$, and
\begin{align*}
h_\kappa(t) =(n-1)\frac{sn'_\kappa(t)}{sn_\kappa(t)},
\end{align*}
where $sn_\kappa(t)$ is the unique solution to $\varphi''+\kappa\varphi=0$ with $\varphi(0)=0$  and $\varphi'(0)=1$.
\end{lemma}

Combining Lemma  \ref{Lin} with Lemma \ref{lem4.1}, we have
\begin{lemma} \label{Ricci-controp}
Let $M^n$ be a  complete submanifold  in the Euclidean space $\mathbb{R}^{m}$.  Then for $p> n$, we have
\begin{align}
\int^r_0 \{h(t, \theta)- h_0(t)\}_+^{p}\omega(t,\theta)dt \leq C_2(n,p)
\int_0^r |\Phi|^p \omega(t,\theta)dt,  \label{int-control}
\end{align}
where $C_2(n,p)$ is a constant depending only on $n$ and $p$, and
\begin{align*}
h_0(t) =\frac{n-1}{t}.
\end{align*}
\end{lemma}

Using a method similar to \cite{Ku2}, we have the following volume growth estimate which has its own interest.
\begin{lemma} \label{volume-estimate}
Let $M^n$ be a complete non-compact  submanifold  in  $\mathbb{R}^{m}$.  Assume
\begin{align}
\underset{t \rightarrow \infty}\lim \frac{1}{V(t)}
\int_{B(t)} |\Phi|^p dv=0  \label{total-v2}
\end{align}
for some $p> n$. Then we have
\begin{align*}
\underset{t\rightarrow \infty}\lim \frac{d}{dt}
\log V(t)=0,
\end{align*}
 and for any $\epsilon>0$, there exists $t_0>0$ such that for any $t> t_0$,
 \begin{align*}
 V(t) \leq C(t_0) e^{\epsilon t} V(t_0)
\end{align*}
holds for some constant $C(t_0)>0$. That is, the volume of $M$ grows uniformly sub-exponentially. In particular, for any $\delta \in (0,1)$ and any sequence $\{b_i\}$ with $\underset{i\rightarrow \infty}\lim b_i=\infty$, there exists another sequence $\{a_i\}$ satisfying $V(a_i)=\delta V(b_i)$ and $\underset{i\rightarrow \infty}\lim (b_i-a_i)=\infty$.
\end{lemma}
\begin{proof}
For any $\epsilon>0$, we can take sufficiently large $s>0$ such that $h_0(r)< \epsilon$ for any $r> s$. Integrating the equation
$\partial_r \omega(r,\theta)= h(r, \theta) \omega(r,\theta)$ over $[s, t]\times S^{n-1}$
for any $t>s>0$, using (\ref{int-control}) and the H\"{o}lder inequality  yields
\begin{align*}
&\int_{S^{n-1}} \omega(t,\theta)d \theta -\int_{S^{n-1}} \omega(s,\theta)d \theta \\
&= \int_s^t dr \int_{S^{n-1}} h(r, \theta) \omega(r,\theta)d \theta\\
&= \int_s^t dr \int_{S^{n-1}} (h(r, \theta) -h_0(r)) \omega(r,\theta)d \theta + \int_s^t dr \int_{S^{n-1}} h_0(r) \omega(r,\theta)d \theta\\
&\leq \int_{B(t)} \{h(r, \theta) -h_0(r) \}_+ dv + \epsilon (V(t)- V(s))\\
& \leq V(t)\Big( \frac{1}{V(t)} \int_{B(t)} \{h(r, \theta) -h_0(r) \}_+^p dv    \Big)^{\frac{1}{p}}
    +\epsilon (V(t)- V(s))\\
& \leq C_3(n,p) V(t)\Big( \frac{1}{V(t)} \int_{B(t)} |\Phi|^p dv    \Big)^{\frac{1}{p}}
    +\epsilon (V(t)- V(s)).
\end{align*}
Hence,
\begin{align}
\nonumber \frac{d}{dt} \log V(t) =& \frac{1}{V(t)} \int_{S^{n-1}} \omega(t,\theta)d \theta\\
 \leq & \frac{1}{V(t)} \int_{S^{n-1}}\omega(s,\theta)d \theta
  + C_3(n,p)\Big( \frac{1}{V(t)} \int_{B(t)} |\Phi|^p dv    \Big)^{\frac{1}{p}} +\epsilon. \label{dlogv}
\end{align}
By (\ref{total-v2}) we claim that the volume of $M$ is infinite. If, arguing by contradiction, we assume that the volume of $M$ is finite. Then (\ref{total-v2}) implies that $\Phi=0$, which means that $M$ is a totally umbilical submanifold in $\mathbb{R}^{m}$. Since $M$ is complete non-compact, it follows that $M$ is an affine $n$-plane. This implies that $M$ has infinite volume, which leads to a contradiction. Therefore, considering the hypothesis (\ref{total-v2}) and  the fact that $\frac{d}{dt} \log V(t) \geq 0$, it follows from (\ref{dlogv}) that
\begin{align*}
\underset{t\rightarrow \infty}\lim \frac{d}{dt} \log V(t)=0.
\end{align*}
For any $\epsilon>0$, there exists $t_0 >s$ such that
\begin{align*}
 \frac{d}{dt} \log V(t) \leq \epsilon
\end{align*}
 for all $t>t_0$. Integrating this inequality over $[t_0, t]$ gives
\begin{align}
 \log \frac{V(t)}{V(t_0)} \leq \epsilon(t-t_0). \label{log}
\end{align}
That is,
\begin{align*}
 V(t) \leq e^{\epsilon(t-t_0)} V(t_0).
\end{align*}
Since the volume of $M$ is infinite, for any $\delta \in (0,1)$ and any sequence $\{b_i\}$ satisfying $\underset{i\rightarrow \infty}\lim b_i=\infty$, we can get another sequence $\{a_i\}$ satisfying $\underset{i\rightarrow \infty}\lim a_i=\infty$ and $V(a_i)=\delta V(b_i)$. By (\ref{log}), we have
\begin{align*}
 0< \log \frac{1}{\delta} = \log \frac{V(b_i)}{V(a_i)} \leq \epsilon (b_i -a_i)
\end{align*}
 for any $\epsilon>0$, which implies that $\underset{i\rightarrow \infty}\lim (b_i-a_i)=\infty$.
\end{proof}

\begin{lemma} \label{h-decay}
Let $M^n$ be a complete non-compact  submanifold  in  $\mathbb{R}^{m}$. Suppose that
\begin{align}
\underset{t \rightarrow \infty}\lim \frac{1}{V(t)}
\int_{B(t)} |\Phi|^p dv=0  \label{total-hd}
\end{align}
for some $p>n$. Then for any $\delta \in (0,1)$, there exists two sequences $\{a_i\}$ and $\{b_i\}$ satisfying $\underset{i\rightarrow \infty}\lim a_i=\underset{i\rightarrow \infty}\lim b_i=\infty$ and $V(a_i)=\delta V(b_i)$, such that
\begin{align*}
\underset{i \rightarrow \infty}{\lim} \frac{1}{V(b_i)- V(a_i)} \int_{B(b_i)\setminus B(a_i)}  |h(r, \theta)| dv =0.
\end{align*}
\end{lemma}
\begin{proof}
For any $\delta \in (0,1)$, by Lemma \ref{volume-estimate}, there exists two sequences $\{a_i\}$ and $\{b_i\}$ satisfying $\underset{i\rightarrow \infty}\lim a_i=\underset{i\rightarrow \infty}\lim b_i=\infty$ and $V(a_i)=\delta V(b_i)$.
By Lemma \ref{Ricci-controp} we have
\begin{align}
\nonumber &\frac{1}{V(b_i)-V(a_i)} \int_{S^{n-1}} d\theta \int_{a_i}^{b_i}\{ h(r,\theta)-h_0 \}_{+}^{p}\omega(r,\theta)dr\\
\nonumber \leq & \frac{1}{(1-\delta)V(b_i)} \int_{S^{n-1}} d\theta \int_{0}^{b_i}\{ h(r, \theta) -h_0 \}_{+}^{p}\omega(r,\theta) dr\\
\leq&\frac{C_4(n,p)}{(1-\delta)V(b_i)}\int_{S^{n-1}} d\theta \int_{0}^{b_i}|\Phi|^{p}\omega(r,\theta) dr \label{M1}
\end{align}
for some positive constant $C_4(n,p)$.  By the hypothesis (\ref{total-hd}), the right hand side of (\ref{M1}) tends to zero as $i\rightarrow \infty$.  Using the fact that $\lim_{r\rightarrow \infty} h_0(r) =0$ and  the H\"{o}lder inequality, we have
\begin{align}
\nonumber &\frac{1}{V(b_i)-V(a_i)} \int_{B(b_i)\setminus B(a_i)} \{ h(r, \theta)\}_+ dv  \\
\leq& \Big ( \frac{1}{V(b_i)-V(a_i)} \int_{B(b_i)\setminus B(a_i)} \{ h(r, \theta)\}_+^{p} dv \Big)^{\frac{1}{p}} \rightarrow 0  \label{positive}
\end{align}
as $i \rightarrow \infty$. On the other hand, integrating the equation $ \int_{a_i}^{b_i} \partial_r\omega(r,\theta) dr =\omega(b_i,\theta) -\omega(a_i,\theta)$ over $S^{n-1}$ yields
\begin{align}
\nonumber &\frac{1}{V(b_i)-V(a_i)} \int_{B(b_i)\setminus B(a_i)} h(r,\theta)dv\\
  =& \frac{1}{V(b_i)-V(a_i)} \left(\int_{S^{n-1}}\omega(b_i,\theta) d\theta - \int_{S^{n-1}}\omega(a_i,\theta)d\theta\right). \label{integrate-part}
\end{align}
By Lemma \ref{volume-estimate}, we have
\begin{align*}
\nonumber 0 \leq
&\underset{i \rightarrow \infty}{\lim}\frac{1}{V(b_i)-V(a_i)}\int_{S^{n-1}}\omega(b_i,\theta) d\theta \\
\nonumber =&\underset{i \rightarrow \infty}{\lim}\frac{1}{(1-\delta)}\frac{1}{V(b_i)} \int_{S^{n-1}}\omega(b_i,\theta) d\theta \\
 \leq & \frac{1}{1-\delta}\underset{i\rightarrow \infty}\lim \frac{d}{dt} \log V(b_{i})
 =0.
\end{align*}
Similarly,
\begin{align*}
\nonumber \underset{i \rightarrow \infty}{\lim}\frac{1}{V(b_i)-V(a_i)} \int_{S^{n-1}}\omega(a_i,\theta) d\theta =0.
\end{align*}
Substituting into (\ref{integrate-part}) yields
\begin{align*}
\underset{i \rightarrow \infty}{\lim} \frac{1}{V(b_i)-V(a_i)} \int_{B(b_i)\setminus B(a_i)} h(r, \theta) dv=0. \label{1-positive}
\end{align*}
Combining with (\ref{positive}) gives
\begin{equation*}
\underset{i \rightarrow \infty}{\lim} \frac{1}{V(b_i)-V(a_i)} \int_{B(b_i)\setminus B(a_i)}  |h(r, \theta)| dv =0.
\end{equation*}
\end{proof}

Combining Lemma \ref{r-smooth} with Lemma \ref{h-decay}, we obtain the following lemma concerning the integral decay of the Laplacian of the smooth function $\widetilde{r}(x)$, as constructed in Lemma \ref{r-smooth}.
\begin{lemma}\label{r-decay}
Let $M^n$ be a complete non-compact  submanifold  in  $\mathbb{R}^{m}$. Suppose that
\begin{align*}
\underset{t \rightarrow \infty}\lim \frac{1}{V(t)}
\int_{B(t)} |\Phi|^p dv=0  \label{finite-total}
\end{align*}
for some $p>n$. Then  for any $\epsilon>0$, there exists $i_0>0$ such that for any $i> i_0$, we have
\begin{equation}
\int_{B(b_i)\setminus B(a_i)} |\triangle \widetilde{r}| dv \leq \epsilon V(b_i) \label{int-la}
\end{equation}
where $\{a_i\}$, $\{b_i\}$ are the sequences in Lemma \ref{h-decay}.
\end{lemma}
\begin{proof}
For any $\epsilon>0$, using Lemma \ref{h-decay} and the property of $\eta(r)$ in Lemma \ref{r-smooth},
there exists $r_1>0$ such that for any $a_i>r_1$, we have
\begin{align}
\int_{B(b_i)\setminus B(a_i)} |h(r,\theta)| dv \leq \frac{\epsilon}{3} (V(b_i)-V(a_i))\ \ \mbox{and}\ \ \eta(a_i ) \leq \frac{\epsilon}{3}.   \label{eta}
\end{align}
 Since $h(r,\theta)= \triangle r$ at $M\setminus \{o, \mbox{cut}(o)\}$ and the measure of $\mbox{cut}(o)$ is zero, for any $a_i> r_1$, we get from Lemma \ref{r-smooth} and (\ref{eta}) that
\begin{align}
\nonumber\int_{B(b_i)\setminus B(a_i)} |\triangle \widetilde{r}| dv \leq& \int_{B(b_i)\setminus B(a_i)} |h(r,\theta)|dv +2\eta(a_i)\\
\nonumber\leq& \frac{\epsilon}{3} (V(b_i)-V(a_i))+ \frac{2\epsilon}{3}\\
\leq& \frac{\epsilon}{3} V(b_i)+\frac{2\epsilon}{3}.   \label{int-r}
\end{align}
Since  the volume of $M$ is infinite according to the proof of Lemma \ref{volume-estimate}, there exists $i_0$ such that whenever $i>i_0$,
\begin{align*}
1 \leq  V(b_i).
\end{align*}
Thus,  we conclude from (\ref{int-r}) that
\begin{equation*}
\int_{B(b_i)\setminus B(a_i)} |\triangle \widetilde{r}| dv  \leq \epsilon V(b_i)
\end{equation*}
for any $i>i_0$.
\end{proof}

Define the lower level set $\tilde{B}(t)$ of the smooth function $\tilde{r}$ by
\begin{equation*}
\tilde{B}(t)=\{x \in M, \tilde{r}(x)<t \}.
\end{equation*}
Now, by employing the lemmas presented above, we can proceed to prove our main theorem.
\begin{theorem} \label{thm-2}
Let $M^n$ be a complete non-compact submanifold in $\mathbb{R}^{m}$. Suppose that
\begin{align*}
\underset{t \rightarrow \infty}\lim \frac{1}{V(t)}
\int_{B(t)} |\Phi|^p dv=0  \label{ft}
\end{align*}
for some $p>n$. Then the $L^2$ essential spectrum of $M$ is $[0, +\infty)$.
\end{theorem}
\begin{proof}
Let $\widetilde{r}(x)$ be the smoothing function of $r(x)$ in Lemma \ref{r-smooth}. Let $\{a_i\}$ and $\{b_i\}$ be the sequences in Lemma \ref{h-decay} satisfying $\underset{r \rightarrow \infty}\lim a_i = \infty=\underset{r \rightarrow \infty}\lim b_i$ and $V(b_i)=4V(a_i)$. Choose  positive numbers $c_i$, $d_i$ satisfying
$a_i<c_i<d_i<b_i$ and $V(a_i)=\frac{1}{2}V(c_i)=\frac{1}{3}V(d_i)=\frac{1}{4}V(b_i)$. Then by Lemma \ref{volume-estimate}, we have
$\underset{r \rightarrow \infty}\lim (c_i-a_i) = \infty=\underset{r \rightarrow \infty}\lim (b_i-d_i)$.
For $i$ large, let $R_i =\min\{c_i-a_i, b_i-d_i\}-2>1$ and let $\phi_i$ be a cut-off function defined by
\begin{equation*}
\phi_i(t)= \begin{cases}
0  & \text{$t < \frac{a_i+1}{R_i}$},\\
1  & \text{ $\frac{a_i+1}{R_i}+1 \leq t \leq \frac{b_i-1}{R_i}-1$},\\
0   & \text{ $t > \frac{b_i-1}{R_i} $}
\end{cases}    \label{testf}
\end{equation*}
and $\phi_i \in [0,1]$ and $|\phi_i'| +|\phi_i''|$ is bounded.
For any given $\lambda > 0$, consider
\begin{equation*}
\psi_i= \phi_i \Big(\frac{\widetilde{r}}{R_i} \Big)e^{\sqrt{\lambda}\widetilde{r}\sqrt{-1}}.
\end{equation*}
Then
\begin{align*}
\nonumber \triangle \psi_i+ \lambda \psi_i=&  \Big[ \frac{1}{R_i^2}\phi_i''\big(\frac{\widetilde{r}}{R_i}\big)|\nabla \widetilde{r}|^2 + \sqrt{-1}\sqrt{\lambda}\frac{2}{R_i}\phi'_i\big(\frac{\widetilde{r}}{R_i}\big)|\nabla \widetilde{r}|^2 + \Big( \frac{1}{R_i}\phi'_i\big(\frac{\widetilde{r}}{R_i}\big)\\ &+\sqrt{-1}\sqrt{\lambda}\phi_i \big(\frac{\widetilde{r}}{R_i}\big)\Big)\triangle \widetilde{r} \Big] e^{\sqrt{-1}\sqrt{\lambda}\widetilde{r}}+ \lambda (1 -|\nabla \widetilde{r}|^2)\psi_i.
\end{align*}
Since $1 -|\nabla \widetilde{r}|^2 =\langle\nabla r +\nabla \widetilde{r}, \nabla r-\nabla \widetilde{r} \rangle \leq 4|\nabla \widetilde{r}-\nabla r| $, we have
\begin{align}
|\triangle \psi_i+ \lambda \psi_i |\leq \Big(\frac{C(\lambda)}{R_i} + C(\lambda) |\triangle \widetilde{r}| + C(\lambda)|\nabla \widetilde{r}-\nabla r|\Big)\chi  \label{point-estimate}
\end{align}
for some constant $C(\lambda)>0$, where $\chi$ is the characteristic function of $\tilde{B}(b_i-1)\setminus \tilde{B}(a_i+1)$. By Lemma \ref{r-smooth}, there exists $i_1\geq i_0$ such that for any $i\geq i_1$ we have
\begin{equation*}
\tilde{B}(b_i-1)\setminus \tilde{B}(a_i+1) \subset B(b_i)\setminus B(a_i).
\end{equation*}
For any $\epsilon>0$, by (\ref{int-la}) there exists  $i_2\geq i_1$  such that
\begin{equation}
\int_{B(b_i)\setminus B(a_i)}|\triangle \widetilde{r}| dv \leq \frac{\epsilon}{2C(\lambda)} V(b_i) \label{V-S}
\end{equation}
for any $i\geq i_2$.
Integrating (\ref{point-estimate}) over $M$ and using Lemma \ref{r-smooth} and (\ref{V-S}), we have
\begin{align}
\nonumber \int_M |\triangle \psi_i+ \lambda  \psi_i| dv  \leq & \frac{C(\lambda)}{R_i}(V(b_i)-  V(a_i))+ C(\lambda) \int_{B(b_i)\setminus B(a_i)}|\triangle \widetilde{r}| dv \\
\nonumber &+ C(\lambda)\eta(a_i)\\
\nonumber  \leq& \Big(\frac{C(\lambda)}{R_i} + \frac{\epsilon}{2} \Big) V(b_i)+ C(\lambda)\eta(a_i).
\end{align}
Since the volume of $M$ is infinite and $\eta(a_i)\rightarrow 0$ as $i\rightarrow \infty$,  we can choose $i_3\geq i_2$ such that
\begin{align}
 \int_M |\triangle \psi_i+ \lambda  \psi_i| dv  \leq \epsilon V(b_i) \label{up-volume}
\end{align}
for any $i \geq i_3$. On the other hand, by $R_i =\min\{c_i-a_i, b_i-d_i\}-2$ and
\begin{equation*}
\tilde{B}(b_i-R_i-1)\setminus \tilde{B}(a_i+R_i+1)\supset B(b_i-R_i-2)\setminus B(a_i+R_i+2)
\end{equation*}
for any $i \geq i_3$, we have
\begin{align}
\nonumber \|\psi_i\|_{L^2}^2 &\geq \mbox{vol}\big(\tilde{B}(b_i-R_i-1)\big)-\mbox{vol}\big(\tilde{B}(a_i+R_i+1)\big)\\
\nonumber &\geq V(b_i-R_i-2)-V(a_i+R_i+2)\\
\nonumber &\geq V(d_i)-V(c_i)\\
&=\frac{1}{4}V(b_i).\label{L2}
\end{align}
Hence, the relations (\ref{up-volume}) and (\ref{L2}) imply that
 \begin{align*}
 \int_M |\triangle \psi_i+ \lambda  \psi_i| dv  \leq 4\epsilon \|\psi_i\|_{L^2}^2.
\end{align*}
Since $\lambda>0$ is arbitrary, by Lemma \ref{criteria}
the $L^2$ essential spectrum of $M$ is $[0, +\infty)$.
\end{proof}

By applying  Theorem \ref{thm-2}, we immediately obtain the following result.
\begin{corollary} \label{total-mean}
Let $M^n$ be a complete non-compact submanifold in $\mathbb{R}^{m}$ with $\|A\|_{L^p} < \infty$ for some  $p>n$. Then the $L^2$ essential spectrum of $M$ is $[0, +\infty)$.
\end{corollary}
\begin{proof}
Since $|\Phi|^2 \leq |A|^2$ and $n|H|^2 \leq |A|^2$, we have $\|H\|_{L^p} \leq \frac{1}{\sqrt{n}}\|A\|_{L^p} <\infty$. Thus, it follows from the result of \cite{CMV} that the volume of $M$ is infinite. Therefore,
\begin{align*}
\underset{t \rightarrow \infty}\lim \frac{1}{V(t)}
\int_{B(t)} |\Phi|^p dv\leq \underset{t \rightarrow \infty}\lim \frac{1}{V(t)}
\int_{B(t)} |A|^p dv=0, \label{ft}
\end{align*}
which concludes the proof by Theorem \ref{thm-2}.
\end{proof}

If the second fundamental form of a submanifold $M$ in $\mathbb{R}^{m}$ is bounded, the Gauss equation implies that the Ricci curvature of $M$ has a lower bound. Then Lemma \ref{Stc} and Corollary \ref{total-mean} give the following result.
\begin{corollary}
Let $M^n$ be a complete non-compact submanifold in $\mathbb{R}^{m}$ with bounded second fundamental form and $\|A\|_{L^q} < \infty$ for some  $q>n$. Then the $L^p$ essential spectrum of $M$ is $[0, +\infty)$ for all $p\in [1, +\infty]$.
\end{corollary}

\begin{remark}
For a minimal submanifold $M$ in $\mathbb{R}^{m}$ with $\|A\|_{L^n}< \infty$, it follows from the result of \cite{An} that $\underset{r(x) \rightarrow \infty}{\lim} |A|(x) =0$. Thus the Gauss equation implies that $\underset{x \rightarrow \infty}{\liminf} \mbox{Ric}_M(x)=0$. By the main result of \cite{LZ} the $L^p$ essential spectrum of $M$ is $[0, +\infty)$ for all $p\in [1, +\infty]$.
\end{remark}
 Finally, we have the following spectral Bernstein type result for submanifolds under a curvature pinching assumption.
\begin{corollary}
Let $M^n$ be a complete non-compact submanifold in $\mathbb{R}^{m}$.
 Assume that
\begin{equation*}
  (n-1) |A|^2 \leq n^2|H|^2.
\end{equation*}
 Then the $L^p$ essential spectrum of $M$ is $[0, +\infty)$ for all $p \in [1, +\infty]$.
\end{corollary}
\begin{proof}
By Lemma \ref{Lin}, the assumption $(n-1) |A|^2 \leq n^2|H|^2$ implies that the Ricci curvature of $M$ is nonnegative. Therefore, the proof follows from Theorem 3 of  \cite{Wa}.
\end{proof}

\vspace{0.5cm}
\noindent $\mathbf{Acknowledgements}$.
The authors would like to thank the referees for the very valuable comments and
detailed corrections.\\

\noindent \textbf{Data Availability} Our manuscript has no associate data.\\

\noindent \textbf{Declarations}\\

\noindent \textbf{Competing interests} On behalf of all authors, the corresponding author states that there is no conflict of interest.\\

Yuxin Dong, School of Mathematical Sciences and
Laboratory of Mathematics for Nonlinear Science,
Fudan University, Shanghai, 200433, China.

E-mail address: yxdong@fudan.edu.cn\\

Hezi Lin, School of Mathematics and Statistics \& Key Laboratory of Analytical Mathematics and Applications (Ministry of Education) \& FJKLAMA, Fujian Normal University, Fuzhou 350117, China.

E-mail address: lhz1@fjnu.edu.cn\\

Wei Zhang, School of Mathematics, South China University of Technology, Guangzhou 510641, China.

E-mail address: sczhangw@scut.edu.cn

%\end{CJK*}
\end{document}